\documentclass{amsart}
\usepackage[utf8]{inputenc}
\usepackage{mymacros, stmaryrd, tikz, pstricks, pst-node, pst-text, pst-tree, subfigure, epsfig, psfrag}
\usepackage{float}

\newtheorem{thm}{Theorem}[section]
\newtheorem{cor}[thm]{Corollary}
\newtheorem{lem}[thm]{Lemma}
\newtheorem{prop}[thm]{Proposition}

\theoremstyle{definition}
\newtheorem{defn}[thm]{Definition}
\newtheorem{rem}[thm]{Remark}
\newtheorem{notn}[thm]{Notation}

\renewcommand{\Tr}{\textup{Tr}\,}

\DeclareMathOperator{\intact}{int}
\DeclareMathOperator{\extact}{ext}

\title{The Tutte polynomial and toric Nakajima quiver varieties}
\author{Tarig Abdelgadir}
\address{The Abdus Salam International Centre for Theoretical Physics, 
Stada Costiera 11, 
Trieste 34151, 
Italy
}
\email{tabdelga@ictp.it}

\author{Anton Mellit} 
\address{University of Vienna, Oskar-Morgenstern-Platz 1, Vienna 1090, Austria}
\email{anton.mellit@univie.ac.at}

\author{Fernando Rodriguez Villegas}
\address{The Abdus Salam International Centre for Theoretical Physics, 
Stada Costiera 11, 
Trieste 34151, 
Italy
}
\email{villegas@ictp.it}
\date{\today}

\begin{document}

\maketitle

\begin{abstract}
For a quiver $Q$, we take $\cM$ an associated toric Nakajima quiver variety and $\Gamma$ the underlying graph. In this article, we give a direct relation between a specialisation of the Tutte polynomial of $\Gamma$, the Kac polynomial of $Q$ and the Poincar\'e polynomial of $\cM$.
We do this by giving a cell decomposition of $\cM$ indexed by spanning trees of $\Gamma$ and `geometrising' the deletion and contraction operators on graphs.
These relations have been previously established in Sturmfels-Hausel \cite{hausel-sturmfels} and (Crawley-Boovey)-Van den Bergh \cite{CB-vdB}, however the methods here are more hands-on.
\end{abstract}

\section{Introduction}
The Tutte polynomial packs a number of fundamental numerical graph invariants into a two-variable polynomial.
It features heavily in modern graph theory and illuminates connections between it and other fields.
Here we give a direct geometric argument relating it to important polynomial invariants in representation theory and geometry.
We start by introducing our polynomials of interest and their mutual connections.

The {\it Tutte polynomial} of a graph $\Gamma$ with edge set $E$ and vertex set $V$ is given by:
$$
\bfT_{\Gamma}(x,y)=\sum \nolimits _{{D\subseteq E}}(x-1)^{{k(D)-k(E)}}(y-1)^{{k(D)+\#D-\#V}}, 
$$
where $k(D)$ denotes the number of connected components of the subgraph of $\Gamma$ with edge set $D$. 
Tutte showed that $\bfT_\Gamma$ has non-negative integer coefficients by expressing it as sum over spanning trees $T \subset \Gamma$ as follows:
$$
\bfT_\Gamma(x,y)=\sum_T x^{\intact(T)}y^{\extact(T)},
$$
where $\intact(T)$ and $\extact(T)$ are integral weights attached to spanning tree $T$ that depend on a fixed ordering of the set of edges $E$. Here we will only consider the specialization $\bfT_\Gamma(1,q)$, for which the corresponding formulas simplify to
\begin{equation}\label{eq:formula1}
\bfT_{\Gamma}(1,q)=\sum \nolimits _{{D\subseteq E\;\text{connected}}}(q-1)^{{1+\#D-\#V}}, 
\end{equation}
\begin{equation}\label{eq:formula2}
\bfT_{\Gamma}(1,q)=\sum_{T} q^{\extact(T)}.
\end{equation}

On the other hand, when the edges $\Gamma$ are given an orientation, we may consider the indecomposable representations of the corresponding quiver $Q$.
The polynomial counting the number of indecomposable representations of $Q$ over $\bF_q$ of dimension vector $\bfv$ is called the {\em Kac polynomial}, we will use $A_\bfv(q)$ to denote it.
Throughout we will fix $\bfv$ to be $(1,\ldots,1)$ and call representations with this dimension vector {\it toric representations}.
It is easy to deduce from \eqref{eq:formula1} (see~\cite{hausel-sturmfels}) that the Kac polynomial satisfies
$$
A_\bfv(q)=\bfT_\Gamma(1,q).
$$
So it is natural to ask for an interpretation of \eqref{eq:formula2} in terms of the Kac polynomial.

Last, but no least, is the Poincar\'e polynomial of the Nakajima quiver variety $\cM:= \cM_{\lambda, \theta}(\bfv)$ with $\bfv$ as above and $(\lambda,\theta)$ generic hyperk\"ahler parameters, here $\lambda \in \bk^{Q_0}$ and $\theta \in \bR^{Q_0}$.
Definition~\ref{def:generic} spells out what we mean by generic.
According to general theory developed in \cite{CB-vdB}, the Kac polynomial equals the Poincar\'e polynomial of the associated Nakajima quiver variety $\cM$ up to a power of $q$. The proof however uses deep geometric arguments, and in particular does not produce a direct interpretation of the formula \eqref{eq:formula2}.

In this work, we directly connect both the Kac polynomial and the Poincar\'e polynomial of the Nakajima variety to the Tutte polynomial using the formula \eqref{eq:formula2}.
More precisely, given an ordering on $E$, we show that spanning trees of $\Gamma$ naturally index both: subsets of indecomposable representations of $Q$ and cells in a cell decomposition of $\cM$. For each tree $T$, the number of irreducible representations of $Q$ in the corresponding subset is $q^{\extact(\Gamma, T)}$, while the corresponding cell $\cM_T\subset \cM$ is isomorphic to $\bA_{\bk}^{b_1(Q)+\extact(T)}$.

The Nakajima quiver variety $\cM$ parametrizes stable representations of the double quiver associated with $Q$, which for every edge $e$ of $Q$ contains the opposite edge $e^*$. By forgetting the maps attached to the new edges $e^*$, each point of $\cM$ produces a representation of $Q$. So one can try to construct a cell decomposition of $\cM$ in such a way that a point of $\cM$ belongs to the cell indexed by a particular tree $T$ precisely when the corresponding representation of $Q$ belongs to the subset labelled by $T$. This idea does produce a nice cell decomposition in the case $\theta=0$ and $\lambda$ is generic. However, as soon as $\lambda=0$ and a non-trivial stability condition $\theta$ is involved, this naive approach does not work; for instance, the corresponding representation of $Q$ may fail to be indecomposable. So we need a more subtle construction.

Our proof goes through `geometrising' the deletion/contraction operators on graphs and spanning trees, we explain below.

The external activity of a spanning tree $T$ may be expressed using a deletion/contraction recursion as follows: fix an ordering on $E$, for $e \in E$ the biggest non-loop edge we have:
$$\extact(\Gamma,T) = \begin{cases} \extact(\,{\Gamma\backslash e, T}\,) & \text{if } e \notin T \\ \extact(\,{\Gamma/e, T/e}\,) & \text{if } e \in T,\end{cases}
$$
the base case is when $\Gamma$ has exactly one vertex, we then set $\extact({\Gamma,T})=\#E.$
This expression of $\extact(T)$ has its root in the beautifully efficient expression of the Tutte polynomial through a deletion/contraction recursion.

To `geometrise' we begin by indexing points of $\bfp \in \cM$ by trees.
This indexing process is somewhat delicate: we use the stability parameter $\theta \in \bR^{Q_0}$ to orient a given spanning tree and give an algorithm that would pick the `biggest' amongst those whose oriented arrows are nonzero in $\bfp$.
One of the subtleties here is that $\theta$ may orient a given edge $e$ of $Q$ in one direction, when $e$ is viewed as an edge of a spanning tree $T$, while orienting it in the opposite direction, when $e$ is seen as an edge of a different spanning tree $T'$.
Furthermore, the partial ordering in which the tree assigned to $\bfp$ is the `biggest' is not the lexicographic ordering induced by the ordering on the edges of $\Gamma$.
See Section~\ref{sc:example} for an example in which both of these subtleties are displayed.

This indexing process gives a cell decomposition of $\cM$.
We then define deletion/contraction operators on the tree-indexed cells of $\cM$ to yield isomorphisms:
$$\cM_T \simeq \begin{cases} (\,\cM\backslash e\,)_T \times \bA_\bk^1 & \text{if } e \notin T \\ (\,\cM/ e\,)_{T/e} & \text{if } e \in T.\end{cases}
$$
This is the content of our main theorems, Theorems~\ref{thm:contract} and \ref{thm:delete}. 
The base case is again one where our quiver has one vertex, in that case $\cM = \bA_\bk^{2(\#Q_0)}$.

The number of times the contraction operator is used to get to the base case is $\#Q_0-1$. 
Therefore the number of times deletion is used to get to the base case is $\#Q_1-(\#Q_0-1) - \extact(T)$. This maybe written as $b_1(Q)-\extact(T)$ where $b_1(Q)$ is the first Betti number of $Q$.
We then have that $$\cM_T \simeq \bA_\bk^{2 \extact(T)} \times \bA_\bk^{b_1(Q)- \extact(T)} \simeq \bA_\bk^{b_1(Q)+\extact(T)}.$$
Hence we have the following expression for the Poincar\'e polynomial of $\cM$:
$$P_\cM(q) = q^{b_1(Q)} \cdot \bfT_Q(1,q).$$
This is restated in Corollary~\ref{cor:poincare}.

The contents of this article are organised as follows.
We start by setting up the notation and relevant background in Section~\ref{sc:back}. We then go on to framing the classical results relating the Tutte polynomial and Kac polynomial in our context in Section~\ref{sc:indecomp}.
In Section~\ref{sc:qvariety}, we address the main content of this work; we define a cell decomposition of $\cM$ indexed by trees and use it to relate the Poincar\'e polynomial of $\cM$ to the Tutte polynomial. 
We finish off with a worked example in Section~\ref{sc:example}.

\subsection*{Acknowledgements}
The authors are grateful to Layla Sorkatti for helpful discussions. The second author's work is supported by the projects Y963-N35 and P31705 of the Austrian Science Fund.

\section{Background and notation}\label{sc:back}
\subsection{Quivers}
A quiver $Q$ is specified by two finite sets $Q_0$ and $Q_1$ together with two maps $h, t \colon Q_1 \rightarrow Q_0$. We call the elements of these sets vertices and edges, respectively.  The maps $h$ and $t$ indicate the vertices at the head and tail of each edge.
A nontrivial path in $Q$ is a sequence of edges $p = e_1 \dotsb e_m$ with $h(e_{k}) = t(e_{k+1})$ for $1 \leq k < m$.  We set $t(p) = t(e_{1})$ and $h(p)= h(e_m)$.  
For each $i \in Q_0$ we have a trivial path $\bfe_i$ where $t(\bfe_i) = h(\bfe_i) = i$.  
The path algebra $\bk Q$ is the $\bk$-algebra whose underlying $\bk$-vector space has a basis consisting of paths in $Q$; the product of two basis elements equals the basis element defined by concatenation of the paths if possible or zero otherwise.  
A cycle is a path $p$ in which $t(p) = h(p)$. Throughout we assume that $Q$ is connected.
A spanning tree $T$ is a connected subquiver that contains all the vertices with the minimal number of edges $\#Q_0-1$.
The first Betti number of a quiver is given by $b_1(Q):= \#Q_1 - \#Q_0 + 1$.
For a commutative ring $R$, the $R$-module of functions $Q_0 \rightarrow R$ will be denoted $R^{Q_0}$.
The double quiver $\Qbar$ associated to $Q$ is the quiver given by adjoining an extra edge of the opposite orientation for each edge $e \in Q_1$, that is $\Qbar_0 = Q$ and $\Qbar_1= Q_1 \cup (Q^\text{op})_1$.
The edge of $\Qbar$ corresponding to the opposite of $e \in Q_1$ will be called $e^*$.

Given a non-loop edge $e \in Q_1$ we define the contracted quiver $Q/e$ as follows. The vertices are given by $(Q/e)_0 = (Q_0 \setminus \{t(e),h(e)\}) \cup \{\iota\}$ and the edges by $(Q/e)_1 = Q_1 \setminus \{e\}$. 
Let $\alpha$ be the inclusion $(Q/e)_1$ in $Q_1$ and $\beta$ be the natural map $\beta \colon Q_0 \rightarrow (Q/e)_0$ taking both $t(e),h(e)$ to $\iota$. 
The head and tail maps from $h',t' \colon (Q/e)_1 \rightarrow (Q/e)_0$ are given by pre-composing $h,t$ on $Q$ with $\alpha$ and post-composing with $\beta$.
A spanning tree $T \subset Q$ naturally defines a spanning tree $T/e$ of $Q/e$ for any given non-loop edge $e\in Q_1$.
Furthermore, given an element $\lambda \in \bk^{Q_0}$, we define $\lambda/e \in \bk^{(Q/e)_0}$ to take the value $\lambda(i)$ at $i \in (Q/e)_0 \setminus \{\iota\}$ and $(\lambda/e)(\iota):=\lambda(h(e))+\lambda(t(e))$.
Given $\theta \in \bR^{Q_0}$ we define $\theta/e \in \bR^{{Q/e}_0}$ in a similar fashion.
We will drop the contraction notation `$/e$' from subquivers, subtrees, $\bk^{Q_0}$, $\bR^{Q_0}$ and their corresponding elements when the contraction is clear from the context.
For $e \in Q_1$, we will use the notation $Q \backslash e$ to denote the quiver $Q$ with the edge $e$ deleted.

A representation $\bfx = (V_i, x_e)$ of $Q$ consists of a vector space $V_i$ for each $i \in Q_0$ and a linear map $x_e \colon V_{t(e)} \rightarrow V_{h(e)}$ for each $e \in Q_1$. 
The dimension vector of $\bfx$ is the integer vector $(\dim V_{i})_{i\in Q_0}$.  
A map between representations $\bfx = (V_i, x_e)$ and $\bfx' = (V_i', x_e')$ is a family $\xi_{i} \colon V_i^{\,} \rightarrow V_i'$ for $i \in Q_0$ of linear maps that are compatible with the structure maps, that is $x_e' \circ\xi_{t(e)} = \xi_{h(e)} \circ x_e$ for all $e \in Q_1$.  
With composition defined componentwise, we obtain the abelian category of representations of $Q$ denoted rep$_\bk(Q)$. 
This category is equivalent to the category $\bk Q$-mod of finitely generated left modules over the path algebra.

Given a dimension vector $\bfv$, a $\theta \in \bZ^{Q_0}$ for which $\bfv \cdot \theta=0$ defines a stability notion for representations of $Q$ with dimension vector $\bfv$.
A representation $\bfx$ is $\theta$-semistable if, for every proper, nonzero subrepresentation $\bfx' \subset \bfx$, we have $\sum_{i \in Q_0} \theta_i \cdot \dim(V_i') \geq 0$.  
The notion of $\theta$-stability is obtained by replacing $\geq$ with $>$. 
For a given dimension vector $\bfv \in \bZ_{Q_0}$, a family of $\theta$-semistable quiver representations over a connected scheme $S$ is a collection of rank $\alpha_i$ locally free sheaves $\cV_i$ together with morphisms $\cV_{t(e)} \rightarrow \cV_{h(e)}$ for every $e\in Q_1$. 
When every $\theta$-semistable representation is $\theta$-stable and the dimension vector is primitive this moduli problem is representable by a scheme $\cM_\theta(Q, \bfv)$, see Proposition 5.3 in \cite{King}.

\subsection{Nakajima quiver varieties}\label{ssc:back}
A more detailed introduction to the ideas below may be found in Ginzburg \cite{Ginzburg} and Kuznetsov \cite{Kuznetsov}.

Suppose we are given a quiver $Q$ and a dimension vector $\mathbf{v}\in \bZ^{Q_0}$. Pick two further vectors $\lambda\in\bk^{Q_0}$, $\theta\in\bR^{Q_0}$ called the hyperk\"ahler parameters. The hyperk\"ahler parameters are required to satisfy $\bfv\cdot \theta=0$, $\bfv\cdot \lambda=0$.
Take $V_i$ to be a vector space of dimension $\bfv_i$ for each $i\in Q_0$. We will use $\cR(Q)$ to denote the space 
$$\cR(Q)=\bigoplus_{e \in Q_1}\big(\text{Hom}(V_{t(e)}, V_{h(e)}) \oplus \text{Hom}(V_{h(e)}, V_{t(e)})\big).$$
A point in $\cR(Q)$ defines a representation of $\Qbar$ with dimension vector $\bfv$.
The vector space $\cR(Q)$ has a natural symplectic structure: it is the cotangent space of $$\bigoplus_{e \in Q_1}\text{Hom}(V_{t(e)}, V_{h(e)}).$$
Change of basis gives a Hamiltonian group action of $G_\mathbf{v}:= \oplus_{i \in Q_0} \GL(V_i)$ on $\cR(Q)$. 
This induces a moment map $\mu \colon \cR(Q) \rightarrow \frakg_\bfv^*$ given by $$(\bfx,\bfx^*)(\bfz) \mapsto  \Tr([\bfx,\bfx^*]\,\bfz).$$

An element $\lambda \in \bk^{Q_0}$ gives an element of $\frakg_\bfv^*$ taking $(x_i)_{i\in Q_0}$ to $\sum_{i\in Q_0} \lambda_i \cdot \text{Tr}\, x_i$.
The set of such elements defines a subset $\bk^{Q_0} \subset \frakg_\bfv^*$ which coincides with the fixed point set of the coadjoint action of $G_\bfv$ on $\frakg_\bfv^*$.
For $\lambda \in \bk^{Q_0} \subset \frakg_\bfv^*$ the closed subset $\mu^{-1}(\lambda)$ is given by $(\bfx,\bfx^*) \in \cR(Q)$ that satisfy the following equations 
\begin{equation}\label{eq:forz}
\sum_{\{e \in Q_1\colon h(e)=i\}} x_e x_e^* -  \sum_{\{e \in Q_1\colon t(e)=i\}} x_e^* x_e\,=\,\lambda_i \qquad i\in Q_0.
\end{equation}
Applying trace to both sides of \eqref{eq:forz} and summing over all $i$, we see that for these equations to have a solution it is necessary that $\bfv\cdot \lambda=0$.

Let $G_m$ be the multiplicative group diagonally embedded in $G_\bfv$. Since $G_m$ acts trivially on $\cR(Q)$, we have an action of $G_\bfv/G_m$ on $\cR(Q)$. The stability parameter $\theta \in \bR^{Q_0}$, because of the condition $\bfv\cdot \lambda=0$, defines a GIT stability condition for this action. GIT stability is equivalent to a more intrinsic King-like stability condition. A point $(\bfx,\bfx^*) \in \cR(Q)$ is $\theta$-semistable with respect to the $G_\bfv/G_m$-action on $\cR(Q)$ if and only if the quiver representation $\Qbar$ given by $(V_i, \bfx,\bfx^*)$ of $\Qbar$ is $\theta$-semistable. We will use $\cR(Q)^\theta$ to denote the $\theta$-semistable points in $\cR(Q)$. The Nakajima quiver variety corresponding the above data is then $$\cM_{\lambda, \theta}(\bfv) := \mu^{-1}(\lambda)\sslash_\theta G_\bfv.$$ The closed subset $\mu^{-1}(\lambda) \subset \cR(Q)$ will be denoted by $Z_\lambda$ to lighten the notation. 

We assume that the pair $(\lambda,\theta)$ is generic in the following sense:
\begin{defn}\label{def:generic}
The hyperk\"ahler parameters $\lambda\in\bk^{Q_0}$, $\theta\in\bR^{Q_0}$ are generic for a dimension vector $\bfv\in \bZ^{Q_0}$ if $\bfv\cdot \lambda=0$ and $\bfv\cdot\theta=0$ hold, but for any dimension vector $\bfv'\in \bZ^{Q_0}$ satisfying $0\leq \bfv'_i \leq \bfv_i$ for all $i\in Q_0$ we have that $\bfv'\cdot \lambda=0$ and $\bfv'\cdot\theta=0$ implies $\bfv'=\bfv$ or $\bfv'=0$.
\end{defn}
This clearly guarantees that any semistable point of $\cM_{\lambda, \theta}(\bfv)$ is stable, so GIT quotient above is a nice quotient and the Nakajima variety is smooth .

We will assume throughout that $\bfv=(1,\ldots,1)$ and drop $\bfv$ from the notation.

\begin{rem}
There is a version of Nakajima quiver varieties which involved framing and that further depends on a second dimension vector $\bfw\in\bZ^{Q_0}$. However, by Crowley-Boevey's trick explained on p. 11 in \cite{Ginzburg}, these varieties are isomorphic to the varieties without framing for the quiver obtained from $Q$ by adding a single new vertex $\infty$ with dimension $1$ and $\bfw_i$ edges from $\infty$ to $i$ for each $i\in Q_0$. Hence our results apply to the case of framed Nakajima varieties for $\bfv=(1,\ldots,1)$ and $\bfw$ arbitrary.
\end{rem}

\begin{rem}
If the pair $(\lambda,\theta)$ is generic, but $\theta$ alone happens to be non-generic, one can perturb $\theta$ without changing the set of semistable points to make $\theta$ generic. So we will assume throughout that $\theta$ is generic, which means that for all $\bfv'\in \bZ^{Q_0}$ satisfying $0\leq \bfv'_i \leq \bfv_i$ for all $i\in Q_0$ we have that $\bfv'\cdot\theta=0$ implies $\bfv'=\bfv$ or $\bfv'=0$.
\end{rem}

\section{Counting indecomposable representations}\label{sc:indecomp}
The content of this section is classical, we restate the results using our notation for context.

Fix a quiver $Q$.
The aim in this section is to count indecomposable representations of $Q$ with dimension vector $\bfv:= (1,\ldots,1)$.

\begin{defn}
Given a representation $\bfx = (V_i, x_e)$ of $Q$, let $$D:=\{e\in Q_1 \,|\, x_e \textup{ is not an isomorphism}\}.$$   The {\em inversion graph} $K_\bfx$ of $\bfx$ is the subgraph $Q - D$.
\end{defn}

\begin{lem}\label{lm:coninvgr}
A representation $\bfx$ is indecomposable if and only if its inversion graph $K_\bfx$ is connected.
\end{lem}

We aim to count indecomposable representations of $Q$ over the field $\mathbb{F}_q$.
This will be a sum over spanning trees of $Q$.
Given Lemma \ref{lm:coninvgr}, we can assume our quiver $Q$ is connected.

First, we define the {\em external activity} of a given spanning tree $T \subset Q$ in a recursive fashion and denote it $\extact(Q,T)$.
To do so we require an ordering on the non-loop edges in $Q_1$.
Fix one once and for all.
Let $e \in Q_1$ be the biggest non-loop edge in our ordering. Then
$$\extact(Q,T) = \begin{cases} \extact(\,{Q\backslash e, T}\,) & \text{if } e \notin T \\ \extact(\,{Q/e, T/e}\,) & \text{if } e \in T.\end{cases}
$$
When $\# Q_0 =1$ we set $\extact(Q,T)= \# Q_1.$
Note that $Q\backslash e$ and $Q/e$ in the above statement naturally inherit an ordering on their non-loop edges.
We will drop the $Q$ from the notation $\extact(Q,T)$ when the quiver is clear from the context.

One may similarly define internal activity in a recursive fashion: we apply the recursion to non-bridge edges with the base case being a quiver $Q$ for which every edge is a bridge and $\intact(T)=\#Q_1$.
We will not spell this out here since we will not need it.

\begin{prop}\label{prop:euler}
Given a quiver $Q$ and a spanning tree $T \subset Q$ we have that $\extact(Q,T) \leq b_1(Q)$.
\end{prop}

\begin{proof}
The number of times the contraction operator is used to get to the base case is $\#Q_0-1$. 
Therefore the number of times deletion is used to get to the base case is $\#Q_1-(\#Q_0-1) - \extact(T)= b_1(Q)-\extact(T)$.
This is therefore non-negative.
\end{proof}

To count representations, we associate a tree to each representation $\bfx$ as follows. Let $e\in Q_1$ again be the biggest non-loop edge. If $e\in K_x$, we can use the linear map attached to $e$ to identify $V_{h(e)}$ with $V_{t(e)}$. So $\bfx$ corresponds to an indecomposable representation $\bfx/e$ of $Q/e$. The correspondence $\bfx \leftrightarrow \bfx/e$ is one-to-one. If on the other hand $e\notin K_x$, we can think of $\bfx$ as an indecomposable representation of $Q\backslash e$. Proceeding in this way some edges of $Q$ will be contracted, some edges will be deleted and some edges will become loops. Let $T$ be the tree formed by the contracted edges. Then the number of loops left at the end of the recursion is $\extact(T)$. Representations associated to a given tree $T$ are in bijection with representations of a quiver with one vertex and $\extact({T})$ loops, so their number is $q^{\extact({T})}$. We have proved

\begin{thm}\label{thm:numindecom}
The number of indecomposable representations of $Q$ over $\mathbb{F}_q$ is given by the sum over spanning trees $T \subset Q$ below: $$A_\bfv(q)=\sum_{T\subset Q}\, q^{\extact({T})}.$$
\end{thm}

\begin{rem}
The external activity of a given tree depends on the ordering we chose above. 
However, the statement of Theorem \ref{thm:numindecom} implies that the number of indecomposables does not.
\end{rem}

The formula \eqref{eq:formula2} now relates the Kac polynomial to the Tutte polynomial

\begin{cor}
The polynomial $A_\bfv(q)$ is equal to the specialisation $\bfT_Q(1,q)$ of the Tutte polynomial.
\end{cor}

\section{Cell decomposition of the quiver variety}\label{sc:qvariety}

Let $\cM$ be $\mathcal{M}_{\lambda, \theta}(\bfv)$ as defined in Subsection \ref{ssc:back}. We will denote a general point in $\cM$ by $\bfp = (\bfx, \bfx^*)$. 
The aim here is to give a cell decomposition of $\cM$, expressing its class in the Grothendieck ring of varieties in terms of the class of the affine line. 
This will be done in a similar recursive fashion to the indecomposable representations count in Section \ref{sc:indecomp}.
In particular, the count will be over spanning trees and will use contraction and deletion operators.

\subsection{Contraction/deletion}
We start by setting up the contraction language for elements of $\cR(Q)$.

\begin{defn}
For $\bfp \in \cR(Q)$, let $$D:=\{e\in Q_1 \,|\, x_e \textup{ and } x_e^* \textup{ are not isomorphisms}\}.$$   
The {\em inversion graph} $K_\bfp \subset Q$ of $\bfp$ is then $Q \backslash D$.
\end{defn}

Take $\bfp \in \cR(Q)$ and assume $e\in K_\bfp$ is not a loop.
Without loss of generality we take $x_e$ to be the isomorphism.
We define a point $\bfp/e$ in $\cR(Q/e)$ using the following recipe.
Set the vector space at the new vertex $\iota$ to be the graph of $x_e$, i.e.\ $V_\iota := \{(v, x_e(v)) \in V_{t(e)} \oplus V_{h(e)} \, | \, v \in V_{t(e)}\}$.
Vector spaces at other vertices remain unchanged.
The vector space $V_\iota$ is naturally isomorphic to both $V_{t(e)}$ and $V_{h(e)}$.
We use these isomorphism to associate linear maps $x_{e'}$ and $x_{e'}^*$ for every $e' \in (Q/e)_1$ whose corresponding edge in $Q$ is incident to either $t(e)$ or $h(e)$.
Linear maps for the other edges are clear.

The following lemmas confirm that that the contraction of $\bfp$ behaves well with respect to the hyper-K\"ahler parameters $(\lambda, \theta)$.

\begin{lem}\label{lm:contralambda}
Take $e \in K_\bfp$ not a loop. If $\bfp \in Z_\lambda$ then $\bfp/e \in (Z/e)_{(\lambda/e)}.$
\end{lem}

\begin{proof}
We should check equations (\ref{eq:forz}) for $Q/e$.
The only nontrivial check is at the vertex $\iota$.
Without loss of generality, assume $x_e$ is the isomorphism.
First conjugate the equation corresponding to $t(e)$ with the isomorphism $V_{t(a)} \rightarrow V_\iota$ and that corresponding to $h(e)$ with the isomorphism $V_\iota \rightarrow V_{h(e)}$.
Taking the sum of the conjugates kills off the term corresponding to $e$ and the result follows.
\end{proof}

\begin{lem}\label{lem:contraction preserves stability}
Take $e\in K_\bfp$ not a loop. If $\bfp \in \cR(Q)$ is $\theta$-stable then $\bfp/e \in \cR(Q/e)$ is $(\theta/e)$-stable.
\end{lem}

\begin{proof}
It is easier to see the contrapositive.
Assume $\bfp/e$ is $(\theta/e)$-unstable.
Let $\bfq/e \subset \bfp/e$ be a destabilising submodule.
To lift $\bfq/e$ to $\cR(Q)$ it suffices to specify the vector spaces at $t(e)$ and $h(e)$.
We take them to be $V_{t(e)}$ and $V_{h(e)}$ respectively, if the vector space defined by $\bfq/e$ at $\iota$ is full dimensional.
If the vector space defined by $\bfq/e$ at $\iota$ is $0$ we take both of them to be $0$.
This lift then de-stabilises $\bfp$.
\end{proof}

\begin{rem}
There is an ambiguity in the above construction if both $x_e$, $x_e^*$ are isomorphisms.
However, in that case, both choices give equivalent representations and they will descend to the same point in $\cM$.
\end{rem}

The deletion operation is easier to define.
For a point $\bfp \in \cR(Q)$ and any edge $e \in Q$. Then ignoring the linear maps $x_e$ and $x_e^*$ gives a representation $\bfp \backslash e \in \cR(Q\backslash e)$. If either $x_e=0$ or $x_e^*=0$ then $\bfp \in Z_\lambda$ implies $\bfp \backslash e \in (Z\backslash e)_\lambda$. Observe that $\theta$-stability is not always preserved under this operation. To summarize, we have the following cases:

\begin{enumerate}
    \item Both $x_e$ and $x_e^*$ are non-zero. We can contract $e$ using $x_e$ or $x_e^*$ and obtain a $\theta/e$-stable representation.
    \item Exactly one out of $x_e$, $x_e^*$ is non-zero. Suppose $x_e\neq 0$. We can contract $e$ or delete $e$. Contraction will always produce a $\theta/e$-stable representation, but deletion sometimes destroys stability.
    \item Both $x_e$ and $x_e^*$ are zero. Deleting $e$ produces a $\theta\backslash e$-stable representation.
\end{enumerate}

These observations are the ideas behind Notation \ref{notn:intree} below. 
Before we get there, we address spanning trees in the Nakajima quiver varieties setting.

\begin{lem}\label{lm:coninvgr2}
Take $\theta \in \bR^{Q_0}$: if $\bfp \in \cR(Q)$ is $\theta$-stable then $K_\bfp$ is connected.
\end{lem}

\begin{proof}
Assume $K_\bfp$ is not connected. Choose a connected component and let $J \subset Q_0$ be the vertices of this component. 
Let $\delta$ be the indicator function for $J \subset Q_0$, then elements $(t^{\delta(i)}\cdot 1_i)_{i \in Q_0} \in  G_\bfv$ stabilise $\bfx$ for any $t \in \bG_m$.
This gives a positive dimensional stabiliser subgroup and contradicts $\theta$-stability.
One may also see the lemma by decomposing the corresponding quiver representation into two direct summands and using King stability.
\end{proof}

Lemma \ref{lm:coninvgr2} above implies that given a generic $\theta \in \bR^{Q_0}$ the inversion graph contains a spanning tree.
For a specific choice of generic $\theta$ more can be said.

\subsection{Stability and trees}
In a spanning tree $T \subset Q$, every edge $e \in Q$ splits $T$ into two connected components call them $T_{t(e)}$ and $T_{h(e)}$.
If $\sum_{i \in T_{t(e)}} \theta_i < \sum_{j \in T_{h(e)}}\theta_j$ we take $e$ to be oriented as in $Q$.
If $\sum_{j \in T_{h(e)}} \theta_i < \sum_{i \in T_{t(e)}}\theta_j$ we take $e$ to be reverse oriented.
We may then view $T$ as a subquiver of the double quiver $\Qbar$. 

This orientation maybe equivalently defined in slightly different language.
The incidence homomorphism $\textup{inc} \colon \bR^{Q_1} \rightarrow \bR^{Q_0}$ is defined by $\chi_e \mapsto \chi_{h(e)}-\chi_{t(e)}$.
The image of $\text{inc}$ lies is the hyperplane of $\bR^{Q_0}$ consisting of vectors $\theta$ satisfying $\theta\cdot \bfv=0$. The inc-images of the edges of a spanning tree of $T \subset Q$ define a basis of this hyperplane. That is, a spanning tree decomposes the stability space $\{\theta\in \bR^{Q_0}:\theta\cdot \bfv\}$ into $2^{|Q_0|-1}$ simplicial cones.
A generic stability parameter $\theta$ lies in precisely one of these cones.
The cone to which $\theta$ belongs then defines an orientation on our spanning tree $T$.
The discussion above inspires the following definitions.

\begin{notn}\label{notn:orientedtree}
Fix $\theta \in \bR^{Q_0}$ a generic stability parameter and let $T$ be a spanning tree of $Q$. 
We will write $T^\theta$ for the oriented spanning tree of $\Qbar$ defined by the $\theta$ induced orientation.

Define the weight of $e\in T^\theta$ by
$$
\theta_{e,T} = \sum_{j\in T_{h(e)}} \theta_j>0.
$$
\end{notn}
It is not hard to check the identity
$$
\textup{inc}\left(\sum_{e\in T^\theta} \theta_{e,T}\; \chi_e\right) = \theta.
$$

\begin{lem}[Orientation is preserved under contraction]
Fix $\theta \in \bR^{Q_0}$ a generic stability parameter and let $T$ be a spanning tree of $Q$.
For $e \in T^\theta$, we have that $(T^\theta)/e \subset \overline{(Q/e)}$ is the $\theta/e$ oriented spanning tree of the spanning tree $T/e \subset Q/e$.
Here we abuse notation and refer to the edge in $Q$ and a corresponding edge in $\Qbar$ by the same symbol $e$. Moreover, we have $\theta_{e',T}=(\theta/e)_{e',T/e}$ for each edge $e'\in T$ different from $e$.
\end{lem}

\begin{proof}
Straightforward.
\end{proof}

\begin{defn}
Take $\bfp \in \cR(Q)$ and let $$\Dbar :=\{e\in \Qbar_1 \,|\, x_e \text{ is not an isomorphism}\}.$$   
The {\em oriented inversion graph} $K_{\bfp}^o \subset \Kbar_\bfp \subset \Qbar$ of $\bfp$ is then $\Qbar \backslash \Dbar$.
\end{defn}

\begin{lem}\label{lm:orientedtree}
Let $\theta \in \bR^{Q_0}$ be a generic stability parameter and take $\bfp \in \cR(Q)$. 
The point $\bfp$ is $\theta$-stable if and only if there exists a subtree $T\subset Q$ such that $T^\theta \subset K_\bfp^o$.
\end{lem}

\begin{proof}
Assume $\bfp$ is $\theta$-stable and that $T^\theta \subset K_\bfp^o$ does not exist.
Take a tree $T \subset K_\bfp$. We say an edge $e \in T^\theta$ is \emph{faulty} if $e\notin K_\bfp^o$. Pick a faulty edge $e$.
Let $Q_{t(e)}$ and $Q_{h(e)}$ be the subsets of $Q_0$ formed by the vertices of $T_{t(e)}$ and $T_{h(e)}$ respectively. We may assume that there are no arrows $a \in \Qbar_1$ from $Q_{t(e)}$ to $Q_{h(e)}$ whose corresponding linear map $x_{a}$ is an isomorphism, otherwise we replace $e$ with $a$ and get a tree with one less faulty arrow.
Setting $U_{e'} = V_{e'}$ for $e' \in (Q_{t(e)})_0$ and $U_{e'}=0$ for $e' \in (Q_{h(e)})_0$ gives a destabilising subrepresentation of $\bfp$.

Now assume $\bfp$ is unstable and $T$ as above exists. Without loss of generality we can assume that all the arrows outside of $T^\theta$ are zero. Let $\bfq$ be a destabilizing subrepresentation. If $\bfq$ is decomposable then one of its direct summands is also destabilizing. So we can assume that $\bfq$ is indecomposable, which means that the subgraph $K$ of $T^\theta$ formed by the vertices where $\bfq$ is not zero is connected. The complement $T^\theta\setminus K$ does not have to be connected, denote its connected components by $K_1,\ldots,K_m$. We have that there is precisely one edge in $T^\theta$ from $K_i$ to $K$ for each $i$ and no edges between $K_i$ and $K_j$ for $i\neq j$. The decomposition of $T^\theta\setminus K$ into connected components corresponds to a decomposition of the quotient $\bfp/\bfq$ into a direct sum of $m$ representations. One of these direct summands is a destabilizing quotient representation, call it $\bfr$. Suppose it is supported on subgraph $K_i$. Set $\bfq'=\ker(\bfp\to\bfr)$. We have that $\bfq'$ is supported on $K'$, which is formed by the vertices of $K$ and $K_j$ for $j\neq i$. Now $K_i$ and $K'$ decompose the tree $T^\theta$ into two parts with a single edge connecting them which goes from $K_i$ to $K'$. On the other hand, since $K_i$ is a destabilizing quotient, we have $\sum_{v\in K_i} \theta_v>0>\sum_{v\in K'} \theta_v$, which by the construction of $T^\theta$ implies that the edge connecting $K_i$ and $K'$ must be oriented from $K'$ to $K_i$, a contradiction.
\end{proof}

We now want to associate a tree to each $\theta$-stable point $\bfp \in \cR(Q)$.
The idea here is to pick the `biggest' spanning tree $T$ for which $T^\theta \subset K_\bfp^o$.
Such $T$ exists by Lemma \ref{lm:orientedtree} since $\bfp$ is $\theta$-stable.
However, this partial ordering in which this tree is `biggest' is harder to pin down.
It is easier to describe it through the smallest edge.
We will need the following notation and lemma.

\begin{notn}\label{notn:theta'}
Take $\bfp \in \cR(Q)^\theta$, $e\in Q_1$ and assume $\bfp \backslash e$ is $\theta$-unstable; here $e$ should be thought of as the smallest edge in some ordering.
Every destabilising subrepresentation of $\bfp \backslash e$ must contain either the head or tail of $e$ but not both, otherwise it would destabilise $\bfp$.
Furthermore, if one destabilising representation of $\bfp \backslash e$ contains $t(e)$ then they all must: if two destabilising subrepresentations were to contain both $t(e)$ and $h(e)$ respectively, then their sum or their intersection must be destabilizing too, which leads to a contradiction.
Without loss of generality, we may assume that all destabilising subrepresentations of $\bfp \backslash e$ contain $t(e)$ and do not contain $h(e)$.
The set of all destabilising representations, call it $\cD$, is finite and nonempty. Note that this implies that $x_e$ is an isomorphism, otherwise any destabilizing subrepresentation would remain a subrepresentation of $\Qbar$.
We take $$\beta_e := \text{min}\{\,\theta(\bfq) \,\,|\,\, \bfq \in \cD \,\}.$$ 
We also fix $\theta':= \theta - \beta_e \,\text{inc}(\chi_e)$.
\end{notn}

\begin{lem}\label{lm:sstabledecomp}
Take $\bfp \in \cR(Q)^\theta$, $e\in Q_1$ and assume $\bfp \backslash e$ is $\theta$-unstable.
Fix $\theta'$ as in Notation \ref{notn:theta'}.
The representation $\bfp \backslash e$ is strictly $\theta'$-semistable.
Furthermore, $\bfp \backslash e$ is an extension of two $\theta'$-stable representations $\bfp_1$ and $\bfp_2$.

\end{lem}

\begin{proof}
Take $\beta_e$ and $\cD$ as in Notation \ref{notn:theta'}.
Minimality of $\beta_e$ gives that elements of $\cD$ are positive when paired with $\theta'$ and so no longer destabilising.
It remains to eliminate the risk of a $\theta$-positive subrepresentation of $\bfp \backslash e$ becoming $\theta'$-non-positive: if such a subrepresentation was to exist, it would have to contain $t(e)$, and its intersection or sum with a $\theta$-minimising element of $\cD$ would form a $\theta$-non-positive subrepresentation of $\bfp$.
Strictness of semistability follows since $\theta'(\bfq)=0$ for a $\theta$-minimising $\bfq \in \cD$.

For the second statement observe that there is a unique $\theta$-minimising subreprestation in $\cD$: if say $\theta(\bfq_1) =\theta(\bfq_2)=\beta_e$ for some $\bfq_1, \bfq_2 \in \cD$ then $\theta(\bfq_1\cap \bfq_2) + \theta(\bfq_1+ \bfq_2)=\theta(\bfq_1)+\theta(\bfq_2)$ and minimiality of $\beta_e$ implies 
$$\theta(\bfq_1\cap \bfq_2)=\theta(\bfq_1+ \bfq_2)=\beta_e,$$
which contradicts genericity of $\theta$.
We take $\bfp_1$ to be this unique $\theta$-minimising element of $\cD$ and $\bfp_2$ to be the quotient representation $(\bfp \backslash e)/\bfp_1$.

Minimality of $\bfp_1$ implies it is $\theta'$-stable.
For stability of $\bfp_2$ observe that any submodule of it corresponds to a submodule of $\bfp$ in such a way that any destabilising module of $\bfp_2$ gives a $\theta$-minimizing module of $\bfp$.
\end{proof}

\begin{rem}
Lemma \ref{lm:sstabledecomp} maybe alternatively stated as: $\bfp$ is $S$-equivalent to a $\theta'$-polystable representation with two components $\bfp_1$ and $\bfp_2$.
\end{rem}

\begin{rem}
It might seem natural to consider the Harder-Narasimhan filtration of $\bfp \backslash e$ to get a decomposition into smaller representations.
This however does not seem to give the desired results.
\end{rem}

As in Section \ref{sc:indecomp}, we fix an ordering on the edges of $Q$.
\begin{notn}\label{notn:intree}
Take $\bfp \in \cR(Q)^\theta$ and let $e$ be the smallest non-loop edge.
We define a tree $T$ associated to $\bfp$ recursively as follows:
\begin{enumerate}
\item If $\bfp \backslash e$ is $\theta$-stable $e\notin T$.
We then consider $\bfp\backslash e \in \cR(Q \backslash e)$.
\item Otherwise take $e \in T$. We use Lemma \ref{lm:sstabledecomp} to give us $\theta'$-stable subrepresentations $\bfp_1$ and $\bfp_2$; we then consider them on their corresponding subquivers. Applying the algorithm for $\bfp_1$ and $\bfp_2$ produces spanning trees in the subquivers, which are then glued together using the edge $e$ to a spanning tree of $Q$.
\end{enumerate}
We stop the algorithm when $Q$ has one vertex.
We will denote the set of all points $\bfp \in \cR(Q)$ associated to a given spanning $T \subset Q$ by $\cR(Q)_T$.
Furthermore, $Z_{\lambda,T}^\theta$ will be $Z_{\lambda} \cap \cR(Q)_T^\theta$.
\end{notn}

\begin{lem}
The algorithm defined in Notation \ref{notn:intree} associates a unique tree $T$ to every $\theta$-stable point $\bfp \in \cR(Q)$.
Furthermore, for $\bfp \in \cR(Q)_T$ we have that $T^\theta \subset K_\bfp^o$.
\end{lem}

\begin{proof}
Existence of $T$ follows from Lemma \ref{lm:orientedtree} and uniqueness follows from the algorithm.
The last statement follows because in notations of Lemma \ref{lm:sstabledecomp} $x_e$ is an isomorphism and $\theta(\bfq_1)=\beta_e$ is negative, so $e$ is oriented the same way as in $T^\theta$.
\end{proof}

The subsets $\cR(Q)_T^\theta$ and $Z_{\lambda,T}^\theta$ are $G_\bfv$-invariant since $\theta$-stability is a $G_\bfv$-invariant property.
We will use $\cM_T \subset \cM$ to denote the quotient of $Z_{\lambda,T}^\theta$ by the $G_\bfv$-action. 

For $e\in K_\bfp$ a non-loop edge we introduce the notation $\cM/e$ for the Nakajima quiver variety on the contracted quiver $Q/e$ with dimension vector $\bfv= (1,\ldots,1)$, and hyperk\"ahler parameters $\lambda/e$ and $\theta/e$.
We use $\cM \backslash e$ in an analogous fashion.

\begin{thm}\label{thm:contract}
Let $T$ be a spanning tree of $Q$ and let $e$ be the biggest non-loop edge in $Q$. 
If $e \in T$ then $$\mathcal{M}_T \simeq (\mathcal{M}/e)_{T/e}.$$
\end{thm}

\begin{proof}
Without loss of generality we may assume $T^\theta=T$, this in particular implies that $x_e$ is an isomorphism.
We have a morphism $\cM_T \rightarrow \cM/e$ given by $\bfp \mapsto \bfp/e$. 
In fact the image of this morphism lies in $(\cM/e)_{T/e}$. If we are in case (1) of Notation~\ref{notn:intree} on $Q$, i.e.\ the smallest edge $e'$ is such that $\bfp\backslash e'$ is stable, Lemma~\ref{lem:contraction preserves stability} implies that $(\bfp/e)\backslash e'$ is also stable. Therefore, when applying the algorithm to $\bfp\backslash e$ we are in case (1) as well. If we are applying (2), $e\in T$ implies that both endpoints of $e$ belong to $\bfp_1$ or $\bfp_2$. Hence $\bfp_1/e$ is a destabilizing subrepresentation for $(\bfp/e)\backslash e'$, and so the corresponding step of the algorithm for $\bfp/e$ produces the same decomposition of $Q_0$, and we can continue by recursion. 

On the other hand, a point of $(\cM/e)_{T/e}$ lifts to a unique point of $\cM$ by the following construction. Take $\bfq \in (Z/e)_{\lambda,T}^\theta$, the aim is to provide a lift $\bfp \in Z_{\lambda,T}^\theta$.
We start by defining linear maps for every edge in $\Qbar$.
The linear maps associated to edges not incident to $t(e)$ or $h(e)$ are clear.
For all non-$e$ edges incident to $t(e)$ and $h(e)$, we choose isomorphisms $\phi \colon V_{t(e)} \rightarrow V_\iota$ and $\psi \colon V_\iota \rightarrow V_{h(e)}$ to unwind their corresponding linear maps and define $x_e$ to be $\psi \circ \phi$.
This choice of $\phi$ and $\psi$ will come out in the wash when we quotient by the action of $G_\bfv$.
The only ambiguity left is the linear map associated to $x_e^*$.
We use the equation corresponding to $t(e)$ from (\ref{eq:forz}) to read off $x_e^*$: post-multiplying the equation by $x_e^{-1}$ reduces the $e$-term to $x_e^*$.
One could have equally well used the equation corresponding to $h(e)$ to yield the same result.
Up to the choice of isomorphisms $\phi$ and $\psi$, we have a point $\bfp \in Z_\lambda$ and since $x_e$ is an isomorphism by construction we have $T\subset K_\bfp^o$ which implies $\bfp \in Z_{\lambda}^\theta$ by Lemma \ref{lm:orientedtree}. 

It remains to show that the tree associated to $\bfp$ is precisely $T$. Suppose this is not the case and denote the corresponding tree by $T'\neq T$. Running the algorithm for $\bfp$ produces $T'$ and running the algorithm for $\bfq$ produces $T/e$. Let us consider the first step where the algorithm for $\bfp$ deviates from the algorithm for $\bfq$ and let $e'$ be the smallest edge.
We consider the following three possibilities.

Suppose $e'\notin T'$, $e'\in T$. This means that $\bfp\backslash e'$ is stable. By Lemma~\ref{lem:contraction preserves stability}, $(\bfp\backslash e')/e$ is also stable, so we must have $e'\in T$.
Now suppose $e'\in T'$, $e'\notin T$. Since $e'\notin T$, we have $T \subset K^o_{\bfp\backslash e'}$ which implies that $\bfp\backslash e'$ is stable by Lemma~\ref{lm:orientedtree}. This contradicts the fact that $e'\in T'$.

The remaining case is: $e'\in T$, $e'\in T'$ while the decomposition of $\bfp$ into $\bfp_1$ and $\bfp_2$ in (2) of Notation~\ref{notn:intree} is different from the corresponding decomposition of $\bfq$ into $\bfq_1$ and $\bfq_2$. This can only happen if $e$ connects $\bfp_1$ to $\bfp_2$ so that after contraction of $e$, $\bfp_1$ fails to be a subrepresentation of $\bfq$. In this case, we are forced to choose a subrepresentation of $\bfq$ with higher value of $\theta$, i.e.\ $\theta(\bfp_1)<\theta(\bfq_1)$. We apply the same trick as used in the proof of Lemma~\ref{lm:orientedtree}: set all maps of $\bfp$ which are not in $T$ to zero. Denote the resulting representation by $\bfp'$ and the subrepresentation corresponding to $\bfp_1$ by $\bfp_1'$. Applying our algorithm to $\bfp'$ produces a tree contained in $T$, so it must be $T$. In particular, the first step of the algorithm produces a subrepresentation which goes to $\bfq_1$ after contraction of $e$, so we have  $\theta(\bfq_1)\leq\theta(\bfp_1')=\theta(\bfp_1)$, a contradiction.

It remains to act by $G_\bfv$ to remove the ambiguity of the choice of $\phi$ and $\psi$ and descend to a morphism $(\cM/e)_{T/e} \rightarrow \cM_T$.
Checking the constructed morphisms are mutual inverses is left to the reader.
\end{proof}

Let $T$ be a spanning tree of $Q$ and $\theta\in \bR^{Q_0}$ a generic stability parameter.
For every edge $e \notin T$, we let $C(T,e)$ be the unique cycle of the graph obtained by adding $e$ to $T$.
Assume $e \notin T$ and let $e \neq a \in C(T,e)$, the orientation of $a \in T^\theta$ then defines a direction around $C(T,e)$ and so an orientation on the edge $e$.
We call this the {\em $a$-induced orientation on e}.

\begin{lem}\label{lm:xe=0}
Let $\theta\in \bR^{Q_0}$, $T$ be a spanning tree of $Q$, $e$ be the biggest non-loop edge in $Q$ and $\bfp=(\bfx,\bfx^*) \in \cM_T$ .
Take $e \notin T$ and let $a$ be the smallest edge in $C(T,e)$.
Assume the $a$-induced orientation on $e$ is opposite to its orientation as an edge of $Q$ then $x_e=0$, otherwise $x_e^*=0$.
\end{lem}

\begin{proof}
We may assume $a$ is the smallest edge so $\bfp \backslash a$ is $\theta$-unstable.
If $x_e \neq0$ then considering the spanning tree $S:= (T\backslash a) \cup e$ would give a tree for which $S^\theta \subset K_\bfp^o$.
Lemma \ref{lm:orientedtree} then gives the result.
\end{proof}

\begin{thm}\label{thm:delete}
Let $T$ be a spanning tree of $Q$ and $e$ be the biggest non-loop edge in $Q$.
If $e \notin T$ then $$\mathcal{M}_T \simeq (\mathcal{M}\backslash e)_T \times \bA_\bk^1.$$
\end{thm}

\begin{proof}
We may now adopt a similar strategy to Proof of Theorem~\ref{thm:contract}.
Take $\bfq \in (Z \backslash e)_{\lambda,T}^\theta$, the aim is to provide a lift $\bfp \in Z_{\lambda,T}^\theta$.
The linear maps associated to all edges that are not $e$ may be read off from $\bfq$.
To finish defining the lift, it remains to fix the linear maps $x_e$ and $x_e^*$.
Lemma~\ref{lm:xe=0} allows us to assume that $x_e=0$.
For the other yet undefined linear map we pick an arbitrary $q \in \Hom(V_{h(e)}, V_{t(e)})$ and set $x_e^*:=q$.
The equations (\ref{eq:forz}) in $Z_{\lambda,T}^\theta$ then follow directly from those in $(Z\backslash e)_{\lambda,T}^\theta$; $\theta$-stability follows from Lemma \ref{lm:orientedtree}; analysing cases (1) and (2) we deduce that our lift lives in $Z_{\lambda,T}^\theta$. This gives a morphism $$(Z \backslash e)_{\lambda,T}^\theta \times \Hom(V_{h(e)}, V_{t(e)}) \longrightarrow Z_{\lambda,T}^\theta.$$
After acting by $G_\bfv$ this descends to a morphism $(\mathcal{M} \backslash e)_T \times \bA_\bk^1 \longrightarrow \mathcal{M}_T.$

Consider $\bfp\in\cM_T$. By Lemma~\ref{lm:xe=0}, we can delete $e$ and obtain a representation $\bfq$ of $Q\backslash e$; this is $\theta$-stable by Lemma~\ref{lm:orientedtree}. The point $\bfp$ is obtained from $\bfq$ by the above lifting construction for a unique value of $q \in \bA^1_\bk$. If $T'$ is the tree associated to $\bfq$, then $T=T'$ follows from the first part of the proof.
\end{proof}

\begin{cor}
For $T\subset Q$ a spanning tree, we have $\cM_T \simeq \bA_\bk^{b_1(Q)+ \extact(T)}$.
\end{cor}

\begin{proof}
First we study $\cM_{\lambda,\theta}(Q, \bfv)$ when $\# Q_0=1$. We have $\lambda=0$, so there is an arbitrary choice of linear maps in $\End(V_0)$ for every edge of the double quiver $\Qbar$, therefore $\cM(\bF_q)=\bA_\bk^{2 (\#Q_1)}$.

We now let $Q$ be a general quiver. 
Proposition \ref{prop:euler} gives that $\extact(T) \leq b_1(Q)$ for a spanning tree $T \subset Q$.
The difference $b_1(Q)-\extact(T)$ is precisely the number of times one uses the deletion operator when reducing $Q$ to a quiver of one vertex in the recursion computing $\extact(T)$.

Theorems \ref{thm:contract} and \ref{thm:delete} along with these two observations give that \begin{equation*}
    \cM_T \simeq \bA_\bk^{2 \extact(T)} \times \bA_\bk^{b_1(Q)- \extact(T)} \simeq \bA_\bk^{b_1(Q)+\extact(T)}.\end{equation*}
This completes the proof.\end{proof}

\begin{cor}\label{cor:poincare}
The Poincar\'e polynomial of $\cM$ is given by
$$P_\cM(q) = q^{b_1(Q)} \cdot \bfT_Q(1,q).$$
\end{cor}




\section{An example: $\widetilde{A}_2$} \label{sc:example}

We go through calculations in Section~\ref{sc:qvariety} in an example.
The example is relatively simple yet it exhibits most of the phenomena of interest.

Our starting quiver $Q$ will be that of the affine Dynkin diagram of type $\widetilde{A}_2$.
We label the vertices and edges of $Q$ as in Figure~\ref{fig:A_2graph}.
The figure also contains the three spanning trees of $Q$ which are given by forgetting one of the three edges.
We pick hyperk\"ahler parameters to be $[\lambda,\theta] = [0,(-2,1,1)]$ and order the edges $l>m>s$. 
\begin{figure}[H]
\centering
\subfigure[The quiver $Q$]{
\begin{pspicture}(-0.5,-0.5)(2.5,2.232)
          \cnodeput(1,1.732){A}{0} 
          \cnodeput(0,0){B}{1}
          \cnodeput(2,0){C}{2}
          \psset{nodesep=0pt}
          \ncline{-}{A}{B} \lput*{:120}{$m$}
          \ncline{-}{B}{C} \lput*{:U}{$l$}
          \ncline{-}{C}{A} \lput*{:240}{$s$}
                  \end{pspicture}
}
\\
\subfigure[Tree $T_s$.]{
\begin{pspicture}(-0.5,-0.5)(2.5,2.232)
          \cnodeput(1,1.732){A}{0} 
          \cnodeput(0,0){B}{1}
          \cnodeput(2,0){C}{2}
          \psset{nodesep=0pt}
          \ncline[linecolor=blue]{-}{A}{B}
          \ncline[linecolor=blue]{-}{B}{C}
                  \end{pspicture}
}\quad \quad
\subfigure[Tree $T_m$.]{
\begin{pspicture}(-0.5,-0.5)(2.5,2.232)
          \cnodeput(1,1.732){A}{0} 
          \cnodeput(0,0){B}{1}
          \cnodeput(2,0){C}{2}
          \psset{nodesep=0pt}
          \ncline[linecolor=blue]{-}{B}{C}
          \ncline[linecolor=blue]{-}{C}{A}
                  \end{pspicture}
} \quad \quad
\subfigure[Tree $T_l$.]{
\begin{pspicture}(-0.5,-0.5)(2.5,2.232)
          \cnodeput(1,1.732){A}{0} 
          \cnodeput(0,0){B}{1}
          \cnodeput(2,0){C}{2}
          \psset{nodesep=0pt}
          \ncline[linecolor=blue]{-}{A}{B}
          \ncline[linecolor=blue]{-}{C}{A}
                  \end{pspicture}
}
\caption{The quiver $Q$ and its spanning trees.}\label{fig:A_2graph}
\end{figure}
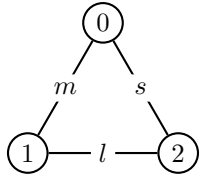
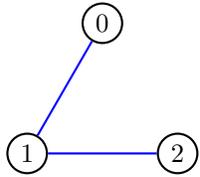
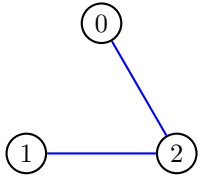
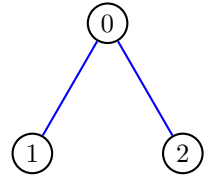

Since our dimension vector $\mathbf{v}$ is $(1,\dots,1)$, every representation in $\cM$ is isomorphic to one where the non-zero vector spaces at the vertices are the vector space $k$.
The linear maps at the arrows are then naturally elements of $\bA^1_\bk$.

Every point in $\cM$ is equivalent to one of the representations displayed in Figure~\ref{fig:exreps}.
The division into subfigures also indicates the cell $\cM_{T_i}$ to which the points belong.
Arrows in {\blue blue} indicate the corresponding oriented tree ${T_i}^\theta \subset \Qbar$.
Note that the orientation of the biggest edge $l$ (or any edge for that matter) may be different for different spanning trees.
The result displayed in Figure~\ref{fig:exreps} gives that $\#\cM(\bF_q)=q^2+q+q=q^2+2q$.
\begin{figure}[H]
\centering
\mbox{
\subfigure[$\cM_{T_s}$: here $(q_1,q_2) \in \mathbb{A}_\bk^2$.]{
\psset{unit=1.3cm}
\begin{pspicture}(-0.4,-0.5)(2.4,3.5)
          \cnodeput(1,1.732){A}{$\bk_0$} 
          \cnodeput(0,0){B}{$\bk_1$}
          \cnodeput(2,0){C}{$\bk_2$}
          \psset{nodesep=0pt}
          \ncarc[offset=4pt]{<-}{A}{B} \lput*{:120}{$-q_1q_2$}
          \ncarc[offset=4pt, linecolor=blue]{<-}{B}{A} \lput*{:300}{1}
          \ncarc[offset=2pt]{<-}{B}{C} \lput*{:U}{$-q_1q_2$}
          \ncarc[offset=2pt, linecolor=blue]{<-}{C}{B} \lput*{:180}{1}
          \ncarc[offset=2pt]{<-}{C}{A} \lput*{:240}{$q_1$}
          \ncarc[offset=2pt]{<-}{A}{C} \lput*{:60}{$q_2$}
                  \end{pspicture}
}

\subfigure[$\cM_{T_m}$: here $q \in \mathbb{A}_\bk^1$.]{
\psset{unit=1.3cm}
\begin{pspicture}(-0.4,-0.5)(2.4,3.5)
          \cnodeput(1,1.732){A}{$\bk_0$} 
          \cnodeput(0,0){B}{$\bk_1$}
          \cnodeput(2,0){C}{$\bk_2$}
          \psset{nodesep=0pt}
          \ncarc[offset=2pt]{<-}{A}{B} \lput*{:120}{$q$}
          \ncarc[offset=2pt]{<-}{B}{A} \lput*{:300}{0}
          \ncarc[offset=2pt, linecolor=blue]{<-}{B}{C} \lput*{:U}{1}
          \ncarc[offset=2pt]{<-}{C}{B} \lput*{:180}{0}
          \ncarc[offset=2pt, linecolor=blue]{<-}{C}{A} \lput*{:240}{1}
          \ncarc[offset=2pt]{<-}{A}{C} \lput*{:60}{0}
                  \end{pspicture}
}

\subfigure[$\cM_{T_l}$: here $q \in \mathbb{A}_\bk^1$.]{
\psset{unit=1.3cm}
\begin{pspicture}(-0.4,-0.5)(2.4,3.5)
          \cnodeput(1,1.732){A}{$\bk_0$} 
          \cnodeput(0,0){B}{$\bk_1$}
          \cnodeput(2,0){C}{$\bk_2$}
          \psset{nodesep=0pt}
          \ncarc[offset=2pt]{<-}{A}{B} \lput*{:120}{0}
          \ncarc[offset=2pt, linecolor=blue]{<-}{B}{A} \lput*{:300}{1}
          \ncarc[offset=2pt]{<-}{B}{C} \lput*{:U}{$q$}
          \ncarc[offset=2pt]{<-}{C}{B} \lput*{:180}{0}
          \ncarc[offset=2pt, linecolor=blue]{<-}{C}{A} \lput*{:240}{1}
          \ncarc[offset=2pt]{<-}{A}{C} \lput*{:60}{0}
                  \end{pspicture}
}}
\caption{Cell decomposition of $\cM$.}\label{fig:exreps}
\end{figure}

In Figure~\ref{fig:delcontra} we go through the algorithm in Notation~\ref{notn:intree} for a given point $\bfp\in \cM$.
We will indicate the smallest edge using the colour {\green green}.
The steps in Figure \ref{fig:delcontra} give that the tree $T_l$ labels the point $\bfp$ given there.
Note that $T_m^\theta \subset K_\bfp^o$.
We remark that the algorithm chose $T_l$ even though it is lexicographically smaller than $T_m$. 

\begin{figure}[H]
\centering
\mbox{
\subfigure[Our point $\bfp$.]{
\psset{unit=1.3cm}
\begin{pspicture}(-0.4,-0.5)(2.4,2.232)
          \cnodeput(1,1.732){A}{$\bk_0$} 
          \cnodeput(0,0){B}{$\bk_1$}
          \cnodeput(2,0){C}{$\bk_2$}
          \psset{nodesep=0pt}
          \ncarc[offset=2pt]{<-}{A}{B} \lput*{:120}{0}
          \ncarc[offset=2pt]{<-}{B}{A} \lput*{:300}{1}
          \ncarc[offset=2pt]{<-}{B}{C} \lput*{:U}{1}
          \ncarc[offset=2pt]{<-}{C}{B} \lput*{:180}{0}
          \ncarc[offset=2pt, linecolor=green]{<-}{C}{A} \lput*{:240}{1}
          \ncarc[offset=2pt, linecolor=green]{<-}{A}{C} \lput*{:60}{0}
                  \end{pspicture}
}

\subfigure[Decomp.\ $\theta'=(-1,1,0)$.]{
\psset{unit=1.3cm}
\begin{pspicture}(-0.4,-0.5)(2.4,2.232)
          \cnodeput[fillstyle=solid, fillcolor=yellow](1,1.732){A}{$\bk_0$} 
          \cnodeput[fillstyle=solid, fillcolor=yellow](0,0){B}{$\bk_1$}
          \cnodeput[fillstyle=solid, fillcolor=purple](2,0){C}{$\bk_2$}
          \psset{nodesep=0pt}
          \ncarc[offset=2pt]{<-}{A}{B} \lput*{:120}{0}
          \ncarc[offset=2pt, linecolor=blue]{<-}{B}{A} \lput*{:300}{1}
          \ncarc[offset=2pt]{<-}{B}{C} \lput*{:U}{1}
          \ncarc[offset=2pt]{<-}{C}{B} \lput*{:180}{0}
                  \end{pspicture}
}
\subfigure[The result.]{
\psset{unit=1.3cm}
\begin{pspicture}(-0.4,-0.5)(2.4,2.232)
          \cnodeput(1,1.732){A}{$\bk_0$} 
          \cnodeput(0,0){B}{$\bk_1$}
          \cnodeput(2,0){C}{$\bk_2$}
          \psset{nodesep=0pt}
          \ncarc[offset=2pt]{<-}{A}{B} \lput*{:120}{0}
          \ncarc[offset=2pt, linecolor=blue]{<-}{B}{A} \lput*{:300}{1}
          \ncarc[offset=2pt]{<-}{B}{C} \lput*{:U}{1}
          \ncarc[offset=2pt]{<-}{C}{B} \lput*{:180}{0}
          \ncarc[offset=2pt, linecolor=blue]{<-}{C}{A} \lput*{:240}{1}
          \ncarc[offset=2pt]{<-}{A}{C} \lput*{:60}{0}
                  \end{pspicture}
}

}
\caption{Notation \ref{notn:intree} algorithm}\label{fig:delcontra}
\end{figure}

We now change the ordering so that $s>l>m$ and examine the decomposition of the $\cM$ into cells indexed under this ordering.
This cell decomposition is displayed in Figure~\ref{fig:exreps2}.

\begin{figure}[H]
\centering
\mbox{
\subfigure[$\cM_{T_s}$: here $q \in \mathbb{A}_\bk^1$.]{
\psset{unit=1.3cm}
\begin{pspicture}(-0.4,-0.5)(2.4,2.232)
          \cnodeput(1,1.732){A}{$\bk_0$} 
          \cnodeput(0,0){B}{$\bk_1$}
          \cnodeput(2,0){C}{$\bk_2$}
          \psset{nodesep=0pt}
          \ncarc[offset=3pt]{<-}{A}{B} \lput*{:120}{0}
          \ncarc[offset=3pt, linecolor=blue]{<-}{B}{A} \lput*{:300}{1}
          \ncarc[offset=2pt]{<-}{B}{C} \lput*{:U}{0}
          \ncarc[offset=2pt, linecolor=blue]{<-}{C}{B} \lput*{:180}{1}
          \ncarc[offset=2pt]{<-}{C}{A} \lput*{:240}{0}
          \ncarc[offset=2pt]{<-}{A}{C} \lput*{:60}{$q$}
                  \end{pspicture}
}

\subfigure[$\cM_{T_m}$: here $(q_1,q_2) \in \mathbb{A}_\bk^2$.]{
\psset{unit=1.3cm}
\begin{pspicture}(-0.4,-0.5)(2.4,2.232)
          \cnodeput(1,1.732){A}{$\bk_0$} 
          \cnodeput(0,0){B}{$\bk_1$}
          \cnodeput(2,0){C}{$\bk_2$}
          \psset{nodesep=0pt}
          \ncarc[offset=4pt]{<-}{A}{B} \lput*{:120}{$q_2$}
          \ncarc[offset=3pt]{<-}{B}{A} \lput*{:300}{$q_1$}
          \ncarc[offset=2pt, linecolor=blue]{<-}{B}{C} \lput*{:U}{1}
          \ncarc[offset=2pt]{<-}{C}{B} \lput*{:180}{$-q_1q_2$}
          \ncarc[offset=6pt, linecolor=blue]{<-}{C}{A} \lput*{:240}{1}
          \ncarc[offset=6pt]{<-}{A}{C} \lput*{:60}{$-q_1q_2$}
                  \end{pspicture}
}

\subfigure[$\cM_{T_l}$: here $(q_1,q_2) \in \mathbb{A}_\bk^2$.]{
\psset{unit=1.3cm}
\begin{pspicture}(-0.4,-0.5)(2.4,2.232)
          \cnodeput(1,1.732){A}{$\bk_0$} 
          \cnodeput(0,0){B}{$\bk_1$}
          \cnodeput(2,0){C}{$\bk_2$}
          \psset{nodesep=0pt}
          \ncarc[offset=2pt]{<-}{A}{B} \lput*{:120}{0}
          \ncarc[offset=2pt, linecolor=blue]{<-}{B}{A} \lput*{:300}{1}
          \ncarc[offset=2pt]{<-}{B}{C} \lput*{:U}{0}
          \ncarc[offset=2pt]{<-}{C}{B} \lput*{:180}{$q$}
          \ncarc[offset=2pt, linecolor=blue]{<-}{C}{A} \lput*{:240}{1}
          \ncarc[offset=2pt]{<-}{A}{C} \lput*{:60}{0}
                  \end{pspicture}
}}
\caption{Cell decomposition of $\cM$: different ordering.}\label{fig:exreps2}
\end{figure}
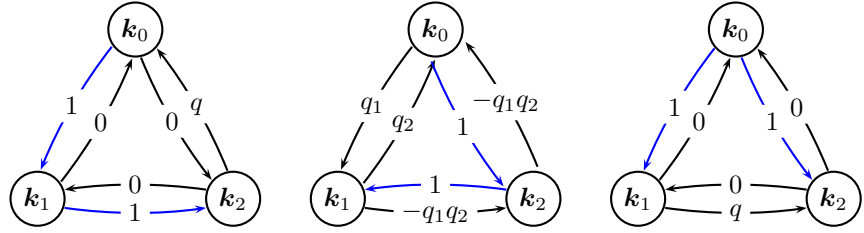

The decompositions displayed in Figures \ref{fig:exreps} and \ref{fig:exreps2} show that a fixed representation may be in labelled by different trees for different orderings.
However, they both give decompositions $\cM = \bA_\bk^2 \sqcup \bA_\bk^1 \sqcup \bA_\bk^1$.
\bibliographystyle{amsplain}
\bibliography{references}

\end{document}